\documentclass[11pt]{article}
\usepackage[a4paper,margin=28mm]{geometry}
\usepackage{amsmath,amssymb,amsthm,mathtools,booktabs,array,longtable}
\usepackage[T1]{fontenc}
\usepackage{lmodern,microtype}
\usepackage[colorlinks=true,linkcolor=blue,citecolor=blue,urlcolor=blue]{hyperref}
\newtheorem{theorem}{Theorem}[section]
\newtheorem{lemma}[theorem]{Lemma}
\newtheorem{proposition}[theorem]{Proposition}
\newtheorem{corollary}[theorem]{Corollary}
\theoremstyle{definition}
\newtheorem{remark}[theorem]{Remark}

\newtheorem*{theoremA}{Theorem A}
\newtheorem*{theoremB}{Theorem B}
\newtheorem*{theoremC}{Theorem C}
\DeclareMathOperator{\rk}{rk}
\DeclareMathOperator{\PG}{PG}
\DeclareMathOperator{\AG}{AG}
\newcommand{\F}{\mathbb F}
\newcommand{\PP}{\mathbb P}
\newcommand{\gb}[2]{\genfrac{[}{]}{0pt}{}{#1}{#2}_{q}}
\newcommand{\gbtwo}[2]{\genfrac{[}{]}{0pt}{}{#1}{#2}_{2}}
\newcommand{\coeff}[1]{[t^{#1}]}

\newcommand{\ind}{\mathbf1}
\newcommand{\papertitle}{Prescribed Tur\'an sign patterns for inverse Kazhdan--Lusztig polynomials of matroids}
\hypersetup{pdftitle={Prescribed Turan sign patterns for inverse Kazhdan--Lusztig polynomials of matroids},pdfauthor={Elias Badalov}}
\title{\papertitle}
\author{Elias Badalov\\\small\href{mailto:ebadal.research@gmail.com}{ebadal.research@gmail.com}}
\date{11 September 2026}
\begin{document}
\maketitle
\begin{abstract}Gao and Xie conjectured that the coefficients of the ordinary,
unnormalized inverse Kazhdan--Lusztig polynomial of every matroid are
log-concave with no internal zeros. We give counterexamples and prove
locality and composition formulas for deletions of projective charts.
For every finite field and every nonempty finite set
$S\subset\{2,3,\ldots\}$, these formulas yield simple $3$-connected
matroids whose negative internal Tur\'an determinants occur exactly at
$S$. Their coefficients are positive and strictly decreasing. The
polynomial degree can be prescribed above $\max S$, and all selected
Tur\'an ratios can tend to zero simultaneously. The universal first-index
inequality remains open.
\end{abstract}
\noindent\textbf{2020 Mathematics Subject Classification:} 05B35, 05A20.
\quad\textbf{Keywords:} matroid; inverse Kazhdan--Lusztig polynomial;
log-concavity; Tur\'an inequality; projective geometry.
\section{Introduction}
Gao and Xie introduced the ordinary inverse Kazhdan--Lusztig polynomial
$Q_M(t)$ and conjectured that its coefficients form a log-concave sequence
with no internal zeros \cite[Conjecture 4.2]{GX}. Write
$Q_M(t)=\sum_{j=0}^K c_jt^j$, where $K$ is the actual degree, and put
\[
 \Delta_j(M)=c_j^2-c_{j-1}c_{j+1}\qquad(1\le j<K).
\]
We give counterexamples for ordinary, unnormalized $Q$. More generally,
we construct simple $3$-connected matroids over any prescribed finite
field for which the negative internal determinants occur at any chosen
nonempty finite set of indices at least two. All other internal
determinants are positive, and the coefficients are positive and strictly
decreasing.

For the first example, take
\[
 V=\F_2^{10},\quad W=\langle e_1,e_2,e_3\rangle,\quad
 X=\langle e_4,\ldots,e_{10}\rangle,\quad H=\langle e_4,\ldots,e_9\rangle,
\]
and use one column for each vector in
\[
 E=(V\setminus\{0\})\setminus
       ((W\setminus\{0\})\cup(X\setminus H)).
\]
This simple rank-ten binary matroid has $952$ elements and
\[
 Q_M(t)=30307166191616+4194859712512t+609014784t^2+89280t^3.
\]
Its second determinant is $-3618068002504704$.
Proposition~\ref{prop:explicit} proves the formula by deletion.
The coefficients are positive and strictly decreasing. Thus this
unimodal sequence fails log-concavity at an internal index.

\subsection{Locality, composition and exact patterns}
Let $\PP(V)$ denote the projective points of a vector space over $\F_q$.
Our first theorem bounds the support of a change in $Q$ under deletion.
\begin{theoremA}[Partial-chart locality]
Let $\dim V=R$, $W<V$, and let $X$ be a hyperplane of $Y\le V$.
For any
$S\subseteq(\PP(Y)\setminus\PP(X))\setminus\PP(W)$ such that
$M=(\PP(V)\setminus\PP(W))\setminus S$ spans $V$, put
$B=\PP(V)\setminus\PP(W)$ and $\ell=\dim(W\cap X)$. Then
\[
 \deg(Q_M-Q_B)\le\min\{\ell+1,\lfloor(R-1)/2\rfloor\}.
\]
The bound $\ell+1$ is sharp for every $\ell\ge0$ and every finite field.
\end{theoremA}
The set deleted may be any part of a chart. The proof bounds the
missing directions in contractions containing the point being deleted;
these contractions need not simplify to complete projective geometries.

To superpose changes, decompose
$V=W\oplus X_0\oplus X_1\oplus\cdots\oplus X_h$,
where $m_i=\dim X_i>0$ for $i\ge1$. Choose arbitrary $L_i\le W$
and $w_i\in W\setminus L_i$, and set
\[
 C_i=\{\langle w_i+l+x\rangle:l\in L_i,\ x\in X_i\},\quad
 N_i=B\setminus C_i,\quad N=B\setminus\bigcup_i C_i.
\]
\begin{theoremB}[Independent-chart composition]
Assume $N$ spans $V$. If $q>2$, then
\[
 Q_N=\sum_{i=1}^h Q_{N_i}-(h-1)Q_B.
\]
For $q=2$, the same formula has an explicit subtractive sum of pair
interactions. The pair $(i,j)$ is supported through degree
$1+\lfloor\dim(L_i\cap L_j)/2\rfloor$. There are no interactions
involving three or more charts. The exact pair formula is given in
Theorem~\ref{thm:composition}.
\end{theoremB}
The external spaces must be independent; the vertices within $W$ may
overlap. The additive formula holds over $\F_4$, since the exceptional
case is $q=2$, not characteristic two.

\begin{theoremC}[Exact strict patterns]
Let $q$ be a prime power, let $\varnothing\ne S\subset\{2,3,\ldots\}$
be finite, put $h=|S|$, and choose $K\ge\max S+1$.
For every integer $n\ge20K^2$, the explicit family of
Theorem~\ref{thm:patterns} is simple, $3$-connected and
$\F_q$-representable, of rank
\[
 R=K+n\left(1+\binom{h+2}{3}\right).
\]
Its inverse polynomial has actual degree $K$ and
$c_0>c_1>\cdots>c_K>0$. At every internal index,
\[
 \Delta_j<0\ \Longleftrightarrow\ j\in S,
 \qquad \Delta_j>0\text{ when }j\notin S.
\]
For all $j\in S$ simultaneously,
$c_j^2/(c_{j-1}c_{j+1})\to0$ as $n\to\infty$.
\end{theoremC}
For each field there is such a represented family; we do not claim one
matroid representable over every field. The displayed rank bound is
sufficient, not optimal. The size formula appears in
\eqref{eq:pattern-size}. Setting $S=\{i\}$ and $K=i+d$ gives any
absolute index $i\ge2$ and any distance $d\ge1$ below the actual top.
Thus failures can lie arbitrarily far from both ends and approach any
prescribed proportion of the coefficient depth. These conclusions use
the actual support, not the maximum degree allowed by rank.

The proof counts surviving traces in ambient subspaces. An incidence
identity reconstructs $Q$ from empty traces and traces whose rank is one
less than their ambient dimension. The only mixed deficient traces come
from pairs of charts over $\F_2$. To prescribe the signs, we choose nested
vertices and strictly convex weights for the external dimensions.
Leading exponents control selected indices, including consecutive ones;
leading coefficients control the remaining indices. Bounds on the
Laurent coefficients make this argument uniform for $n\ge20K^2$ and
also prove strict decrease.

\subsection{Relation to earlier work}
The conjecture concerns ordinary $Q$ \cite[Conjecture 4.2]{GX}.
Xie--Zhang proved it for paving matroids by proving real-rootedness of
the binomially weighted polynomial
$\sum_j\binom{\deg Q_M}{j}c_jt^j$ \cite[Theorems 1.2 and 1.3]{XZ}.
Braden--Ferroni--Matherne--Nepal later gave a counterexample to
real-rootedness of that weighted polynomial \cite[Theorem 5.3]{BFMN}.
Their example does not disprove ordinary-$Q$ log-concavity; see also
their Remark 5.4. We use their inverse deletion formula and KL-basis
deletion maps as algebraic inputs.

Cheng--Liu constructed projective deletions whose ordinary KL polynomials
$P$ are not unimodal over every finite field \cite{CL}. Their Theorem 1.1
evaluates $P$ under quotient-rank hypotheses, Lemma 2.3 describes missing
projective fibers, and Remark 2.7 treats the Bose--Burton geometries used
here. We use these antecedents, but a failure for $P$ does not transfer
to $Q$ by incidence inversion.

The $q$-cone construction and its catenary-data transformations are
established tools \cite{OW,BK}. Valuativity of $Q$
\cite[Theorem 8.8]{AS} and the elementary-split and Schubert formulas
\cite[Theorems 5.3 and 5.9]{FS} give other evaluation frameworks.
Here we evaluate the changes in ordinary $Q$ for projective charts,
including partial charts and binary pair interactions, and use these
formulas to prescribe the signs and sizes of internal Tur\'an ratios.

Sections~\ref{sec:conventions}--\ref{sec:ambient} establish the inputs and
locality mechanism; Sections~\ref{sec:composition}--\ref{sec:patterns}
prove composition and exact realization. The mathematical supplement
contains the comparison examples, thresholds for the complementary
family and first-index positive-class proofs. Section~\ref{sec:first-index-boundary}
proves a stronger first-index inequality for our construction and states
the remaining universal question.

\section{Conventions and algebraic inputs}\label{sec:conventions}
We use $P_M=Q_M=0$ when $M$ has a loop. The polynomial identities and
simplification assertions in this section concern loopless matroids;
in particular, the rank-zero case below is the empty matroid.
For a matroid $M$, write $L(M)$ for its lattice of flats, $M|F$ for
restriction, and $M/F$ for contraction. All represented matroids used in
the constructions are simple. Contractions by flats are loopless, and we
identify their simplifications when describing projective geometries.
The invariants depend only on the ranked lattice of flats, so simplification
does not change them. Ranks are vector dimensions: $\PG(r-1,q)$ has rank $r$.
For a vector space $V$, $\PP(V)$ denotes its projective points.
Put
\[
 g(a,b)=\gb ab,\qquad u_a=q^{\binom a2},\qquad u_0=1,
\]
with $g(a,b)=0$ outside $0\le b\le a$. Gaussian coefficients at $q=2$
will also be written $G_2(a,b)$. For loopless $M$, the characteristic polynomial is
$\chi_M(t)=\sum_{F\in L(M)}\mu(\varnothing,F)t^{r-\rk F}$.

In rank zero, $P_M=Q_M=1$; in positive rank $r$,
$\deg P_M,\deg Q_M<r/2$. The defining identities are
\begin{align}
 t^rP_M(t^{-1})&=\sum_F\chi_{M|F}(t)P_{M/F}(t),\label{eq:P-recursion}\\
 \widehat Q_M(t)&=\sum_Ft^{\rk F}\widehat Q_{M|F}(t^{-1})
                                      \chi_{M/F}(t),\label{eq:Q-recursion}
\end{align}
where $\widehat Q_M=(-1)^rQ_M$ \cite[Theorem 1.3 and equation (14)]{GX}.
Equivalently, $P*\widehat Q$ is the identity in the incidence algebra.
For positive-rank loopless $M$, in particular
\begin{equation}\label{eq:inversion}
 Q_M=\sum_{\varnothing\ne F\in L(M)}
                    (-1)^{\rk F+1}P_{M|F}Q_{M/F}.
\end{equation}
These are characterizations of the same ordinary polynomial, without a
binomial coefficient multiplier or coefficient normalization.
We also use palindromicity of $Z_M=\sum_Ft^{\rk F}P_{M/F}$
\cite[Theorem 2.1]{BFMN}.

For projective geometry, $P=1$ and $Q=u_r$. For example, induction on
rank and Gaussian symmetry show that $P=1$ makes
$Z=\sum_dg(r,d)t^d$ palindromic, so uniqueness gives $P=1$.
Incidence inversion then gives $Q=(-1)^r\mu(0,V)=u_r$.
The latter Moebius value follows from
$\chi_{\PG(r-1,q)}(t)=\prod_{a=0}^{r-1}(t-q^a)$.

Let $\tau(N)=[t^{(\rk N-1)/2}]P_N$ in odd rank, and let $\tau(N)=0$
in even rank. For a simple matroid $N$ and a noncoloop $e$, the deletion
identity of Braden--Ferroni--Matherne--Nepal is
\begin{equation}\label{eq:deletion}
 Q_{N\setminus e}-Q_N=-(1+t)Q_{N/e}
       +\sum_{F\in T_e}\tau((N|F)/e)t^{\rk F/2}Q_{N/F},
\end{equation}
where $T_e=\{F\in L(N):e\in F,\ F\setminus e\notin L(N)\}$
\cite[Theorem 1.4]{BFMN}. Only even-rank flats can contribute. The top
flat is included whenever it satisfies the definition.

\begin{lemma}\label{lem:odd-top}
If $\rk N=2a+1$, then $[t^a]Q_N=[t^a]P_N$.
\end{lemma}
\begin{proof}
Every proper nonempty rank split in $P*\widehat Q$ has product degree
at most $a-1$, by the strict half-rank bounds. The two endpoint terms
at degree $a$ therefore cancel, with signs $+1$ and $-1$.
\end{proof}
This observation is also recorded in \cite[Theorem 4.1]{Vecchi}.

\section{Bose--Burton geometry and an explicit counterexample}
\label{sec:example}
Let $B(r,k;q)$ be the matroid on $\PP(V)\setminus\PP(W)$, where
$\dim V=r$ and $\dim W=k<r$.

\begin{lemma}\label{lem:bose-flats}
Nonempty flats of $B(r,k;q)$ correspond to subspaces $U\not\subseteq W$.
If $d=\dim U$ and $a=\dim(U\cap W)$, their restrictions are $B(d,a;q)$
and their contractions simplify to projective geometry of rank $r-d$.
\end{lemma}
\begin{proof}
Fix $v\in U\setminus W$. Every $w\in U\cap W$ is the difference of
the surviving vectors $v+w$ and $v$, so the surviving trace spans $U$.
Every nonempty flat is its ground-set intersection with its span, proving
the correspondence. A nonzero quotient direction in $V/U$ has normalized
representatives $z+U$. If all were missing, $z+U\subseteq W$ would imply
$U\subseteq W$, a contradiction. Every quotient direction therefore occurs.
\end{proof}

Writing $G_a(t)=\sum_jg(a,j)t^j$, the lemma gives
$Z_B=P_B+G_r-G_k$. Comparing opposite coefficients of $Z$ gives
\begin{equation}\label{eq:boseP}
 [t^j]P_{B(r,k;q)}=g(k,j)-g(k,r-j),\qquad 0\le j<r/2.
\end{equation}
This recovers the formula in \cite[Remark 2.7]{CL}.

\begin{proposition}\label{prop:boseQ}
For $0\le j<r/2$,
\begin{equation}\label{eq:boseQ}
 [t^j]Q_{B(r,k;q)}=
 \sum_{d=2j+1}^{r}(-1)^{d+1}u_{r-d}
 \bigl(g(k,j)g(r-j,d-j)-g(k,d-j)g(r-d+j,j)\bigr).
\end{equation}
\end{proposition}
\begin{proof}
Apply \eqref{eq:inversion} and Lemma~\ref{lem:bose-flats}. For each
rank $d$, extend the sum to all ambient $d$-subspaces: those contained in
$W$ contribute the formal difference $g(d,j)-g(d,d-j)=0$.
These added terms are zero; no rank-$d$ invariant is assigned to an
empty trace.
For $0\le h\le d$, counting pairs $A\le U$ with $A\le W$, $\dim A=h$
and $\dim U=d$ yields
\[
 \sum_{\dim U=d}g(\dim(U\cap W),h)=g(k,h)g(r-h,d-h).
\]
Use $h=j,d-j$ in \eqref{eq:boseP}, and use Gaussian symmetry.
The degree bound gives the lower endpoint $d=2j+1$.
\end{proof}

\begin{proposition}\label{prop:explicit}
The binary matroid in the introduction is simple of rank ten, has $952$
elements, and has
\begin{align*}
 Q_M(t)&=30307166191616+4194859712512t+609014784t^2+89280t^3,\\
 \Delta_2(M)&=-3618068002504704<0.
\end{align*}
\end{proposition}
\begin{proof}
Its $1023-7-64=952$ distinct nonzero columns include the surviving basis
\[
 e_4,e_5,e_6,e_7,e_8,e_9,
 e_1+e_4,e_2+e_4,e_3+e_4,e_1+e_{10}.
\]
This basis persists at every earlier deletion, so all $64$ deletions are
noncoloop deletions. Proposition~\ref{prop:boseQ} gives the starting polynomial
\[
 Q_{B(10,3;2)}=34705212702720+475418034176t+609014784t^2+89280t^3;
\]
the permitted coefficient of degree four is zero.

Suppose $j$ affine points have been removed and $e$ is the next point.
No line through $e$ has both other points in the full missing set:
two points in $W$ sum into $W$, two in the chart sum into $H$, and a
point in $W$ plus a chart point has a nonzero $W$ component unless the
$W$ point is zero. Thus every quotient direction survives and $N/e$
simplifies to $\PG(8,2)$. For every flat $F$ containing $e$, its pointed
restriction $(N|F)/e$ corresponds to a flat restriction of $N/e$.
It has $P=1$, so its $\tau$ vanishes unless $\rk F=2$, including the
top-flat case. The line contractions have $Q=u_8=2^{28}$.

The point $e$ lies on $511$ binary lines. The $7+j$ previously missing
points occupy distinct lines through it. Exactly $504-j$ lines retain
both other points, which is the condition for membership in $T_e$.
Consequently \eqref{eq:deletion} gives
\[
 Q_{N\setminus e}-Q_N=-(1+t)2^{36}+(504-j)2^{28}t,
 \qquad 0\le j\le63.
\]
Summing gives the displayed polynomial. Finally
$609014784^2-4194859712512\cdot89280=-3618068002504704$.
The four coefficients are positive and strictly decreasing.
\end{proof}

The original geometric construction is due to Cheng--Liu, who applied it
to $P$ \cite{CL}; the failed inequality here concerns ordinary $Q$.
For comparison, deleting the shifted chart
$\{w_0+x:0\ne x\in X\}$ instead gives a rank-ten binary matroid on $889$
elements with
\[
 Q=25977839157248+6782309072896t+609014784t^2+89280t^3
\]
and $\Delta_2=-234625546897588224$. An intersecting construction gives
rank-eight examples over every $q>2$, with $3025$ elements at $q=3$.
The complementary family has the sharp thresholds $k\ge3$, $r\ge2k+4$
over $\F_2$ and $r\ge2k+3$ otherwise. Complete proofs and the distinct
characteristic-recursion verifications of these examples are in the
mathematical supplement. These comparisons concern specified families;
no globally minimum rank or size is asserted.

\section{Locality for projective chart deletion}\label{sec:locality}
The linear correction in Proposition~\ref{prop:explicit} is the first case
of a support bound controlled by the intersection of two ambient subspaces.
We begin with the simpler support estimate needed for the proof.

\begin{lemma}\label{lem:deletion-support}
If $U<V$, $\dim U=w$, and $D\subseteq\PP(U)$, then for
$K=\PP(V)\setminus D$ one has $\deg P_K,\deg Q_K\le w$.
\end{lemma}
\begin{proof}
All points outside $U$ remain and span $V$. A subspace $T\not\subseteq U$
is spanned by its surviving trace, and contraction by that trace has
complete projective simplification: an affine coset of $T$ cannot lie in
$U$. If a flat has span $T\le U$ of dimension $d$, all missing quotient
directions lie in $U/T$, of dimension $w-d$.

For $P$, induct on $w$. In the nonempty-flat part of $Z_K$, a flat
spanned inside $U$ contributes in degrees at most $d+(w-d)=w$.
Every other flat contributes only $t^d$, since its contraction has $P=1$.
For $w<j<\rk K/2$, every ambient subspace of dimensions $j$ and
$\rk K-j$ lies outside $U$, and their numbers are equal by Gaussian
symmetry. Comparing these coefficients of the palindromic $Z_K$ forces
$[t^j]P_K=0$. If the interval is empty the half-rank bound already suffices.
The case $w=0$ is projective geometry.

For $Q$, induct on $w$ and then on the number of deleted points.
Delete a next point $e\in U$. It is a noncoloop because the outside-$U$
points still span. Missing directions in $N/e$ are supported in $U/\langle
e\rangle$, so $\deg Q_{N/e}\le w-1$. For a potentially contributing
rank-$2a$ flat $F$ with span $T\not\subseteq U$, the contraction $N/F$
is projective, while $(N|F)/e$ has missing support of dimension at most
$w-1$. The established $P$ bound makes its top coefficient zero unless
$a-1\le w-1$, giving degree at most $a\le w$ after the shift.
If $T\le U$, then $2a\le w$ and $\deg Q_{N/F}\le w-2a$;
the shifted term has degree at most $w-a$. All terms of
\eqref{eq:deletion} have degree at most $w$. The top flat is in the
first case and its rank-zero quotient has $Q=1$.
\end{proof}

\begin{theorem}[Partial-chart locality]\label{thm:locality}
Let $q$ be a prime power, $\dim V=R$, $W<V$, and let $X$ be a hyperplane
of $Y\le V$. Put $C=\PP(Y)\setminus\PP(X)$, and let
$S\subseteq C\setminus\PP(W)$ be arbitrary. Suppose
$M=(\PP(V)\setminus\PP(W))\setminus S$ spans $V$. With
$B=\PP(V)\setminus\PP(W)$ and $\ell=\dim(W\cap X)$,
\[
 \deg(Q_M-Q_B)\le\min\{\ell+1,\lfloor(R-1)/2\rfloor\}.
\]
The bound $\ell+1$ is sharp for every $\ell\ge0$ and every finite field.
\end{theorem}

\begin{lemma}\label{lem:pointed-support}
During any order of these deletions let $e$ be the next point, $N$ the
current matroid, $F$ a flat containing $e$, and $A=\langle F\rangle$,
$d=\dim A$. Put $L=W\cap X$. If $N/F$ has a missing nonzero quotient
direction, then
\[
 A\cap W=A\cap X=A\cap L,\qquad \dim(A\cap L)=d-1.
\]
Every missing quotient direction lies in the image of $L$, of dimension
$\ell-d+1$. Otherwise the quotient simplifies to full projective geometry.
\end{lemma}
\begin{proof}
Choose a representative $e\in Y\setminus(W\cup X)$. All missing nonzero
vectors lie in $W\cup(Y\setminus X)$; this uses all nonzero scalar multiples
of the chart representatives, not just one affine slice. A missing quotient
direction has a wholly missing normalized affine coset $v+A$, not containing
zero. If some $v$ in this coset were outside $Y$, so would $v+e$.
Both would belong to $W$, contradicting $e\notin W$. Hence $v+A\subseteq Y$.
Since $e\in A\setminus X$, its intersection with $X$ is an affine hyperplane
of dimension $d-1$. Its vectors must lie in $W\cap X=L$. Its direction
$A\cap X$ is therefore contained in $L$ and has dimension $d-1$.
Since $e\notin W$, $A\cap W$ is proper in $A$, proving the equalities.
A vector in $(v+A)\cap X$ gives a representative in $L$; rank-nullity
gives the dimension of its quotient image.
\end{proof}

\begin{proof}[Proof of Theorem~\ref{thm:locality}]
If $\ell=R-1$, the half-rank bound suffices. Assume $\ell\le R-2$.
Final spanning ensures that each deletion is a noncoloop deletion.
Lemma~\ref{lem:pointed-support} for $F=\{e\}$ puts the missing directions
of $N/e$ in a proper $\ell$-dimensional subspace. Thus
$\deg Q_{N/e}\le\ell$, and the point term in \eqref{eq:deletion} has
degree at most $\ell+1$.

Consider a rank-$2a$ flat $F$ containing $e$, with span $A$, and let
$\pi:V\to V/\langle e\rangle$. The pointed restriction $(N|F)/e$
is the restriction of $N/e$ to $A/\langle e\rangle$; its missing directions
are supported in
\[
 H=(A/\langle e\rangle)\cap\pi(L),\qquad h=\dim H\le\ell.
\]
If this support is the whole restriction space, $2a-1=h\le\ell$.
If it is proper, Lemma~\ref{lem:deletion-support} bounds the $P$ degree by
$h$, so a nonzero $\tau$ requires $a-1\le h\le\ell$.
In both cases $a\le\ell+1$. If $N/F$ is projective, its shifted term
has degree at most $a$, including a possible top flat.
Otherwise Lemma~\ref{lem:pointed-support} puts its missing directions
in dimension $k=\ell-2a+1$, proper because
$(R-2a)-k=R-\ell-1\ge1$. The shifted degree is at most
$a+k=\ell-a+1\le\ell$. Every deletion correction has the required support.

For sharpness set $R=2\ell+3$, choose $W=X$ of dimension $\ell$ and
a one-dimensional extension $Y$, and delete all of $C$.
Then $B=B(R,\ell;q)$ and $M=B(R,\ell+1;q)$. The former has degree at most
$\ell$ by Lemma~\ref{lem:deletion-support}; by \eqref{eq:boseP} the top
allowed coefficient of $P_M$ is one. Lemma~\ref{lem:odd-top} gives
$[t^{\ell+1}]Q_M=1$, proving equality in the support bound.
\end{proof}

The hypothesis permits arbitrary partial charts. Complete projective point
contractions are sufficient for a linear correction but are not needed for
this theorem: incomplete contractions are controlled by their missing support.
The finite-field argument includes extension fields without modification.

\section{An ambient reconstruction identity}\label{sec:ambient}
Locality bounds the coefficients that may change. To compute the combined
change from several charts, we count the ambient subspaces whose surviving
traces are empty or fail to span by one dimension.
For a spanning restriction $N$ of $\PP(V)$, put
\[
 F(T)=E(N)\cap\PP(T),\quad k(T)=\rk F(T),\quad
 \delta(T)=\dim T-k(T),\qquad T\le V.
\]
The trace $F(T)$ is a flat of $N$, even when it does not span $T$.

\begin{lemma}[Ambient reconstruction]\label{lem:ambient-identity}
Suppose every nonempty trace has $\delta(T)\le1$, and let $E_b$ count
empty $b$-subspaces, including $E_0=1$. For $R=\dim V$, define
\begin{equation}\label{eq:ambient-kernel}
 \mathcal K_{R,b}(t)=
 \sum_{a=0}^{\min(b,R-b)}(-1)^{a+b}g(R-b,a)u_{R-b-a}
              \bigl(t^b+t^a-\ind_{a=b}t^b\bigr).
\end{equation}
Then
\begin{equation}\label{eq:ambient-identity}
 Q_N(t)=\sum_b E_b\mathcal K_{R,b}(t)
 -\sum_{\substack{T\le V:\dim T=2j\\k(T)=2j-1>0}}
           \tau(N|F(T))u_{R-2j}t^j.
\end{equation}
\end{lemma}
\begin{proof}
The first step is a symmetry identity that holds without a deficiency
bound. For every $U\le V$ define
\[
 h_T(t)=(-1)^{k(T)}t^{\delta(T)}Q_{N|F(T)}(t),\qquad
 R_U(t)=\sum_{T\le U}h_T(t),
\]
taking $Q$ of the empty trace to be one. We claim
\begin{equation}\label{eq:ambient-symmetry}
 R_U(t)=t^{\dim U}R_U(t^{-1}).
\end{equation}

Use the standard basis $\Gamma^F$ and the KL basis $\zeta^F$ of
\cite[Section 3, equation (4) and Lemma 3.2]{BFMN}. Their changes of basis
have coefficients $x^{\rk F-\rk G}P_{M|F/G}(x^{-2})$ and
$x^{\rk F-\rk G}(-1)^{\rk F-\rk G}Q_{M|F/G}(x^{-2})$, respectively.
Delete all points of $\PP(V)$ outside $E(N)$ in any order. Every global
deletion preserves rank, since $E(N)$ already spans. In the standard
basis the composite deletion map sends
\[
 \Gamma^T\longmapsto x^{k(T)-\dim T}\Gamma^{F(T)}.
\]
For $d=\dim U$, the projective KL basis vector indexed by $U$ is
$\sum_{T\le U}x^{d-\dim T}\Gamma^T$. After deletion and inverse change
of basis, its coefficient at $\zeta^{\varnothing}$ is
\[
 \sum_{T\le U}(-1)^{k(T)}x^{d+2k(T)-2\dim T}
          Q_{N|F(T)}(x^{-2})=x^dR_U(x^{-2}).
\]
By \cite[Lemma 3.3]{BFMN}, a single deletion sends a KL basis vector whose
support keeps rank to an integer combination of KL basis vectors; if its
support loses rank, it maps to $(x+x^{-1})$ times the corresponding vector.
A vector supported away from the deleted point maps identically.
Every matrix entry, and therefore every composite entry, lies in
$\mathbb Z[x+x^{-1}]$. The displayed coefficient is invariant under
$x\leftrightarrow x^{-1}$, proving \eqref{eq:ambient-symmetry}.
Local rank drops are included; no condition on intermediate trace deficiencies
is used.

Under the deficiency hypothesis, a nonempty trace satisfies
\[
 \deg h_T\le\lfloor(k(T)-1)/2\rfloor+\delta(T)
               \le\lfloor\dim T/2\rfloor.
\]
Hence coefficients of $R_U$ strictly above $d/2$ come only from empty
traces. If $d=2j$, a nonempty term reaches the midpoint only for $T=U$
with $k(U)=2j-1$. Its contribution is
$-[t^{j-1}]Q_{N|F(U)}=-\tau(N|F(U))$, by Lemma~\ref{lem:odd-top}.
Symmetry consequently determines
\begin{align*}
 R_U(t)={}&\sum_{\substack{T\le U:F(T)=\varnothing\\2\dim T\ge d}}
 \bigl(t^{\dim T}+t^{d-\dim T}-\ind_{2\dim T=d}t^{\dim T}\bigr)\\
 &-\ind_{d\text{ even},\ \delta(U)=1,\ k(U)>0}
                          \tau(N|F(U))t^{d/2}.
\end{align*}
Ambient Moebius inversion and $h_V=(-1)^RQ_N$ give
\[
 Q_N=\sum_{U\le V}(-1)^{\dim U}u_{R-\dim U}R_U.
\]
For a fixed empty $b$-space $T$, contributing superspaces have dimensions
$b+a$, with $0\le a\le\min(b,R-b)$, and number $g(R-b,a)$.
Their contribution is \eqref{eq:ambient-kernel}. A deficient midpoint has
even ambient dimension, so its minus sign is unchanged. This proves
\eqref{eq:ambient-identity}.
\end{proof}

The algebraic basis changes and deletion maps are established inputs.
The application of this identity will require a separate classification of
traces. Two computations using the identity share that classification;
agreement between them alone would not prove the geometric hypotheses.

\section{Composition of independent projective charts}
\label{sec:composition}

Chart deletions with independent external directions add when $q>2$.
Over $\F_2$ there is a correction for each pair of charts. The vertices
may intersect, and the origins inside the initially deleted flat may
be chosen independently.

Let
\begin{equation}\label{eq:chart-space}
 V=W\oplus X_0\oplus\cdots\oplus X_h,
 \qquad \dim V=R,\quad \dim W=K,\quad \dim X_i=m_i>0
 \quad(1\le i\le h),
\end{equation}
where \(h\ge1\) and \(X_0\) is allowed to be zero.
Choose \(L_i\le W\) and \(w_i\in W\setminus L_i\), and put
\(\ell_i=\dim L_i\). Define
\[
 C_i=\PP(\langle w_i\rangle+L_i+X_i)\setminus\PP(L_i+X_i).
\]
The points of \(C_i\) have unique representatives \(w_i+l+x\), with
\(l\in L_i\) and \(x\in X_i\). In particular, \(|C_i|=q^{\ell_i+m_i}\)
and \(|C_i\cap\PP(W)|=q^{\ell_i}\). Write
\[
 B=\PP(V)\setminus\PP(W),\qquad
 N_i=B\setminus C_i,\qquad N=B\setminus\bigcup_{i=1}^h C_i
\]
for the corresponding restriction matroids. Throughout this section
we assume that \(N\) spans \(V\). Then \(B\) and each \(N_i\) also span.

\begin{theorem}[Independent-chart composition]\label{thm:composition}
In the construction above, the subspaces \(L_i\) may have arbitrary
intersections. If \(q>2\), then
\begin{equation}\label{eq:nonbinary-composition}
 Q_N(t)=\sum_{i=1}^h Q_{N_i}(t)-(h-1)Q_B(t).
\end{equation}
If \(q=2\), put \(c_{ij}=\dim(L_i\cap L_j)\) and
\(A_b=\AG(b,2)\), of matroid rank \(b+1\). Then
\begin{align}
 Q_N(t)={}&\sum_{i=1}^h Q_{N_i}(t)-(h-1)Q_B(t)\notag\\
 &-\sum_{i<j}(2^{m_i}-1)(2^{m_j}-1)
   \sum_{\substack{0\le b\le c_{ij}\\ b\ {\rm even}}}
   \gbtwo{c_{ij}}{b}2^{\ell_i+\ell_j-2b}
   \tau(A_b)u_{R-b-2}t^{(b+2)/2}.
 \label{eq:binary-composition}
\end{align}
Here \(u_a=2^{\binom a2}\) in the binary formula and \(\tau(A_0)=1\).
There are no interactions involving three or more charts.
The contribution of the pair \(i,j\) is supported in degrees at most
\(1+\lfloor c_{ij}/2\rfloor\).
\end{theorem}

The distinction in the theorem is between \(q=2\) and \(q>2\), rather
than between even and odd characteristic. In particular,
\eqref{eq:nonbinary-composition} holds over \(\F_4\).
The spanning assumption does not imply connectedness for every chart
configuration; we prove a stronger connectivity statement for the
construction used in Section~\ref{sec:patterns}.

\subsection{Projection and deficient traces}

For \(T\le V\), let \(H=T\cap W\), and let \(P\) be the image of \(T\)
in \(X_0\oplus\cdots\oplus X_h\). Call a nonzero vector of \(P\)
\emph{mixed} if it belongs to none of \(X_1,\ldots,X_h\).
This includes every vector with a nonzero \(X_0\)-component.
Every \(H\)-fiber over a mixed vector survives: deleted points outside
\(W\) project into a single \(X_i\).

\begin{lemma}\label{lem:mixed-projection}
If \(q>2\) and \(P\) contains a mixed vector, its mixed vectors span \(P\).
For \(q=2\), the only exception is
\[
 P=\langle x_i,x_j\rangle,\qquad
 0\ne x_i\in X_i,\quad 0\ne x_j\in X_j,\quad i\ne j.
\]
This plane has exactly one mixed nonzero vector, namely \(x_i+x_j\).
\end{lemma}

\begin{proof}
Suppose that the mixed vectors exist but do not span \(P\). Choose a
nonzero functional \(f:P\to\F_q\) that annihilates their span.
The affine hyperplane \(f^{-1}(1)\) is covered by the pure axes \(X_i\).
It spans \(P\), since its differences span \(\ker f\). It therefore
meets at least two axes: if it were contained in one axis, so would \(P\),
contrary to the existence of a mixed vector. Choose
\(u\in X_i\cap f^{-1}(1)\) and \(v\in X_j\cap f^{-1}(1)\), with \(i\ne j\).
For \(q>2\), a scalar \(a\notin\{0,1\}\) makes \(au+(1-a)v\) a mixed
vector in \(f^{-1}(1)\), a contradiction.

For \(q=2\), any third vector \(z\in f^{-1}(1)\) makes \(u+v+z\)
mixed and gives it \(f\)-value one. Indeed, if \(z\) is in a third axis,
three components survive; if it is in one of the first two axes, its
distinctness from \(u,v\) leaves two nonzero components. Thus
\(f^{-1}(1)=\{u,v\}\), so \(\dim P=2\) and the stated description follows.
\end{proof}

Full surviving fibers over vectors spanning \(P\) span \(T\): differences
within one fiber generate \(H\), and projection supplies all of \(P\).
If there are no mixed vectors, \(P\) is zero or lies in one \(X_i\);
otherwise the sum of nonzero vectors in two different axes would be
mixed. Consequently, apart from the binary exceptional planes, every
empty or nonempty deficient trace is contained in \(W\) or in one
\(W+X_i\), and the latter sees only the \(i\)-th chart.

\subsection{The individual-chart classification}

Fix a chart and abbreviate its parameters by \(L,\ell,w,X,m\).
Quotient by \(L\), and write \(W'=W/L\), \(k=\dim W'=K-\ell\),
and \(w'\ne0\) for the image of \(w\). In \(W'\oplus X\) the missing
points are \(\PP(W')\) and the chart represented by \(w'+x\), \(x\in X\).
Their vector supports lie in
\[
 W'\ \cup\ U',\qquad U'=\langle w'\rangle\oplus X.
\]
Inside \(U'\) the surviving points are precisely \(\PP(X)\).
Any subspace outside \(W'\oplus X\) is spanned by its vectors outside
that support, so it cannot have a nonempty deficient trace.

For \(q>2\), a subspace \(A\) contained in neither \(W'\) nor \(U'\)
is spanned by its vectors outside \(W'\cup U'\).
Otherwise an affine hyperplane in \(A\), obtained as \(f^{-1}(1)\)
for a suitable nonzero functional, would be covered by its intersections
with \(W'\) and \(U'\). Neither intersection is the entire affine
hyperplane, since that hyperplane spans \(A\). For \(\dim A=d\ge2\),
each intersection has at most \(q^{d-2}\) points, whereas the affine
hyperplane has \(q^{d-1}\) points. This is impossible when \(q>2\).
The one-dimensional case is immediate. It follows that the nonempty
deficient quotient traces are exactly
\[
 A\le U',\qquad A\not\le X,\qquad d=\dim A\ge2.
\]
The surviving restriction is the full projective geometry on
\(A\cap X\), of rank \(d-1\).

The binary case has one additional possibility.
If a nonempty trace in \(A\) is contained in a hyperplane \(J\), then
\(A\setminus J\) is wholly missing. Its intersections with \(W'\) and
with \(w'+X\) cover it. It cannot lie wholly in \(W'\), since
\(A\setminus J\) spans \(A\). If it lies wholly in \(w'+X\), then
\(J\le X\) and the trace is \(J\setminus\{0\}\), of rank \(d-1\).
In the remaining case two proper affine subspaces cover the affine
\((d-1)\)-space \(A\setminus J\). Each must be a hyperplane, and they
must be disjoint. They have the same direction, of dimension \(d-2\),
contained in \(W'\cap X=0\). Hence \(d=2\).
A nonempty deficient binary plane has two missing points and exactly
one surviving point, again of deficiency one. This proves unit
deficiency in the binary case as well as the nonbinary case.

The only empty quotient subspaces outside \(W'\) are the points
\(\langle w'+x\rangle\), \(x\ne0\), of which there are \(q^m-1\).
Indeed, an empty subspace \(A\not\le W'\) has every outside point in
the chart. Its intersection with \(W'\) must be zero: a nonzero vector
in that intersection either changes the \(W'\)-label of a chart point,
or is a multiple of \(w'\), in which case subtraction produces a
surviving point of \(X\). If \(\dim A\ge2\), two independent points
outside \(W'\), normalized to have \(w'\)-coefficient one, have a
nonzero difference in \(X\), again a survivor. Thus \(\dim A=1\).

Let \(C_d^{(q)}\) count nonempty deficient quotient \(d\)-spaces.
Choosing \(A\cap X\) and then the coset of \(w'\) gives
\begin{equation}\label{eq:single-deficiency-count}
 C_d^{(q)}=\gb{m}{d-1}q^{m-d+1},
 \qquad 2\le d\le m+1,\quad q>2.
\end{equation}
The same formula holds for \(q=2,d\ge3\). For binary planes, count
their unique missing pair. A point in \(W'\setminus\{0\}\) and a chart
point outside \(W'\) give \((2^k-1)(2^m-1)\) pairs. Two distinct chart
points outside \(W'\) give \(\binom{2^m-1}{2}\) further pairs. The third
point survives in each case, and the two classes are disjoint. Therefore
\begin{equation}\label{eq:binary-single-deficiency}
 C_2^{(2)}=(2^m-1)(2^{m-1}+2^k-2).
\end{equation}
All unlisted counts are zero. In particular the \(k\)-dependent term
in \eqref{eq:binary-single-deficiency} is specific to \(\F_2\).

Returning from the quotient by \(L\), fix a projected \(d\)-space
and let its lifted subspace meet \(L\) in dimension \(v\).
Choose that intersection and then a linear graph into its quotient.
The lift multiplicity is
\begin{equation}\label{eq:vertex-lift}
 \gb{\ell}{v}q^{(\ell-v)d}.
\end{equation}
For a nonempty trace, all fibers over surviving quotient points are
present. Their differences generate the vertex intersection, so both
the ambient and surviving ranks increase by \(v\).
Deficiency is preserved. Above a deficient quotient \(d\)-space the
surviving restriction is a projective geometry of vector rank \(d+v-1\)
with a \(v\)-subspace deleted. If \(d+v=2s\), the Bose formula
\eqref{eq:boseP} gives
\begin{equation}\label{eq:single-bose-tau}
 \tau=\gb{v}{s-1}-\gb{v}{s}.
\end{equation}
This also follows directly by comparing opposite coefficients in
\[
 Z(t)=P(t)+\sum_{a=1}^{d+v-1}
       \left(\gb{d+v-1}{a}-\gb{v}{a}\right)t^a.
\]
Every nonempty contraction in this Bose restriction is projective;
the empty flat supplies \(P(t)\) separately.
If \eqref{eq:single-bose-tau} is nonzero, then \(v\ge s-1\), and hence
\(d\le s+1\). Its nonnegativity also follows from Gaussian symmetry
and the adjacent Gaussian ratios, since \(v\le2s-2\).

Lifting the additional empty quotient points gives the extra empty
\(b\)-space count
\begin{equation}\label{eq:single-empty-count}
 E_{i,b}=(q^{m_i}-1)\gb{\ell_i}{b-1}q^{\ell_i-b+1},
 \qquad 1\le b\le\ell_i+1.
\end{equation}
The factor \(q^{m_i}-1\) is not divided by \(q-1\): the coefficient
of \(w_i\) has already been normalized to one.

For the single-chart expression, use \(C_{i,d}^{(q)}\)
from \eqref{eq:single-deficiency-count} and
\eqref{eq:binary-single-deficiency}, with \(k=K-\ell_i\).
Lemma~\ref{lem:ambient-identity}, with the kernel of
\eqref{eq:ambient-kernel}, gives
\begin{align}
 Q_{N_i}-Q_B={}&
 \sum_{b=1}^{\ell_i+1}E_{i,b}\mathcal K_{R,b}\notag\\
 &-\sum_{\substack{2\le d\le m_i+1,\ 0\le v\le\ell_i\\d+v=2s}}
 C_{i,d}^{(q)}\gb{\ell_i}{v}q^{(\ell_i-v)d}
 \left(\gb{v}{s-1}-\gb{v}{s}\right)u_{R-2s}t^s.
 \label{eq:single-chart}
\end{align}
Every kernel and Möbius factor uses the common ambient rank \(R\).
The right side is supported through degree \(\ell_i+1\).
For the empty terms this follows from their kernel degrees; for the
deficient terms it follows from \(s-1\le v\le\ell_i\).
This is also consistent with Theorem~\ref{thm:locality}.

\subsection{Mixed binary traces and the composition law}

Consider an exceptional binary projected plane
\(P=\langle x_i,x_j\rangle\) from Lemma~\ref{lem:mixed-projection},
and write \(H=T\cap W\), \(\dim H=b\).
The full mixed fiber survives and spans rank \(b+1\).
Any surviving point in either pure fiber raises the rank to \(b+2\),
so a nonempty deficient trace occurs exactly when both pure fibers
are wholly missing. This requires
\[
 H\le L_i\cap L_j,
\]
and their graph values modulo \(H\) must belong to \(w_i+L_i\) and
\(w_j+L_j\), respectively. There are
\begin{equation}\label{eq:binary-mixed-count}
 (2^{m_i}-1)(2^{m_j}-1)
 \gbtwo{c_{ij}}{b}2^{\ell_i+\ell_j-2b}
\end{equation}
such subspaces: choose the nonzero external projections, then \(H\),
then the two graph values. These choices are recovered uniquely from
\(T\). Different chart pairs give different projected planes.
The surviving trace is exactly \(\AG(b,2)\).

Its top coefficient has an explicit formula. The rank-\(a\) flats of
\(\AG(2s,2)\) number
\(\gbtwo{2s}{a-1}2^{2s-a+1}\), and every nonempty contraction
simplifies to projective geometry. Comparing the coefficients at
degrees \(s\) and \(s+1\) of its palindromic \(Z\)-polynomial gives
\begin{equation}\label{eq:affine-tau}
 \tau(\AG(2s,2))
   =2^s\left(\gbtwo{2s}{s}
               -2\gbtwo{2s}{s-1}\right)>0.
\end{equation}
For \(s\ge1\), positivity follows because the ratio of the two Gaussian
coefficients is \((2^{s+1}-1)/(2^s-1)>2\).
For \(s=0\) the formula gives one.

\begin{proof}[Proof of Theorem~\ref{thm:composition}]
The trace classification proves that every nonempty ambient trace has
deficiency at most one. Empty traces consist of the base empty spaces
in \(W\) and the additional empty spaces associated with a single chart.
A mixed fiber prevents any mixed trace from being empty.
Nonempty one-axis deficiencies lie in a unique \(W+X_i\), and their
surviving restriction is unchanged by every other chart.
They cannot belong to two such supports, whose intersection is \(W\).

For \(q>2\) this exhausts the empty and deficient traces. Apply
Lemma~\ref{lem:ambient-identity} to \(N,N_1,\ldots,N_h,B\).
Each nonbase contribution occurs once, whereas the base contribution
occurs \(h\) times in the sum of the individual polynomials.
Subtracting \(h-1\) copies gives \eqref{eq:nonbinary-composition}.

For \(q=2\), the only additional traces are those counted in
\eqref{eq:binary-mixed-count}. Their ambient dimension is \(b+2\) and
their surviving rank is \(b+1\). The ambient identity includes them
only when \(b\) is even, with contribution
\(-\tau(A_b)u_{R-b-2}t^{(b+2)/2}\) per trace.
This gives \eqref{eq:binary-composition}. The projection lemma excludes
all higher interactions, and \(b\le c_{ij}\) gives the support bound.
\end{proof}

\begin{remark}\label{rem:linear-binary-interaction}
The binary correction is linear if and only if every \(c_{ij}\le1\).
Indeed, if \(c_{ij}\ge2\), its degree-two contribution has magnitude
\[
 2(2^{m_i}-1)(2^{m_j}-1)\gbtwo{c_{ij}}{2}
      2^{\ell_i+\ell_j-4}u_{R-4}>0,
\]
because \(\tau(\AG(2,2))=2\). All such contributions have the same
negative sign and cannot cancel. Independent external directions
are a separate hypothesis and are retained throughout.
\end{remark}

\section{Exact finite Tur\'an sign patterns}
\label{sec:patterns}

Each chart changes a bounded range of coefficients. We use nested
vertices to choose these ranges, and strictly convex decreasing weights
for the external dimensions. Convexity controls consecutive selected
indices. At unselected indices, the leading coefficients determine the
sign.

\begin{theorem}\label{thm:patterns}
Let \(q\) be a prime power, let
\(\varnothing\ne S=\{s_1<\cdots<s_h\}\subset\{2,3,\ldots\}\) be finite,
and let \(K\ge\max S+1\). Put
\begin{equation}\label{eq:pattern-parameters}
 \ell_a=s_a-2,\qquad
 \lambda_a=\binom{h-a+2}{2}\quad(1\le a\le h),\qquad
 D=1+\sum_{a=1}^h\lambda_a=1+\binom{h+2}{3}.
\end{equation}
For every integer \(n\ge20K^2\), the explicit matroid
\(M=M(q,S,K,n)\) constructed below is simple, \(3\)-connected and
\(\F_q\)-representable, of rank \(R=K+Dn\). Its ordinary inverse
Kazhdan--Lusztig polynomial has degree exactly \(K\), and, writing
\(Q_M(t)=\sum_{j=0}^K c_jt^j\), one has
\[
 c_0>c_1>\cdots>c_K>0.
\]
For every \(1\le j<K\),
\[
 \Delta_j(M)<0\quad\Longleftrightarrow\quad j\in S,
 \qquad
 \Delta_j(M)>0\quad\text{if }j\notin S.
\]
Moreover, for all \(j\in S\) simultaneously,
\begin{equation}\label{eq:simultaneous-ratio-limit}
 \frac{c_j^2}{c_{j-1}c_{j+1}}\longrightarrow0
 \qquad(n\longrightarrow\infty).
\end{equation}
In particular, \(S\) is realized in rank
\begin{equation}\label{eq:pattern-rank-bound}
 K+20K^2\left(1+\binom{h+2}{3}\right).
\end{equation}
\end{theorem}

The rank bound is sufficient, and no minimality assertion is intended.
For each prescribed field the construction gives a matroid over that
field; it does not assert that one matroid has representations over
all finite fields.

\subsection{Construction, size and connectivity}

Take \(W=\F_q^K\) with basis \(e_1,\ldots,e_K\), and set
\[
 L_a=\langle e_1,\ldots,e_{\ell_a}\rangle,\qquad w_0=e_K.
\]
For \(\ell_a=0\), the vertex is zero. Since
\(\ell_a\le K-3\), the origin \(w_0\) belongs to none of the vertices.
Choose independent external spaces \(X_0,\ldots,X_h\), independent
also from \(W\), with
\[
 \dim X_0=n,\qquad \dim X_a=\lambda_an\quad(1\le a\le h).
\]
In \(V=W\oplus X_0\oplus\cdots\oplus X_h\), delete all of \(\PP(W)\)
and, for \(a=1,\ldots,h\), the charts
\begin{equation}\label{eq:pattern-charts}
 C_a=\{\langle w_0+l+x\rangle:l\in L_a,\ x\in X_a\}.
\end{equation}
No chart is deleted on \(X_0\). Use one representative of each remaining
projective point as a matrix column to define \(M(q,S,K,n)\).
The charts are disjoint outside \(\PP(W)\), because their nonzero external
projections lie in different axes. Hence
\begin{equation}\label{eq:pattern-size}
 |E(M)|=\frac{q^R-q^K}{q-1}
       -\sum_{a=1}^h q^{\ell_a}(q^{\lambda_an}-1).
\end{equation}
The columns are distinct projective points, so \(M\) is simple.

\begin{lemma}\label{lem:projective-density-connectivity}
Let \(R\ge3\), and let \(E\) be a set of projective points in an
\(R\)-dimensional \(\F_q\)-space \(V\). Suppose that \(E\) contains
every projective point outside a subspace of codimension at least two.
Then every partition of \(E\) has a spanning side, and the simple
matroid represented by \(E\) is \(3\)-connected.
\end{lemma}

\begin{proof}
The number of retained points is at least
\((q^R-q^{R-2})/(q-1)\). Two distinct ambient hyperplanes have
projective union of size
\[
 \frac{2q^{R-1}-q^{R-2}-1}{q-1};
\]
two coincident hyperplanes have fewer points. The difference is
\[
 \frac{q^{R-1}(q-2)+1}{q-1}>0.
\]
This is positive also for \(q=2\), when it equals one.
Thus \(E\) is not contained in the union of two proper hyperplanes.
It spans \(V\), and if neither side of a partition spanned \(V\),
extending their spans to hyperplanes would give a contradiction.

Write \(\lambda_M(A)=r(A)+r(E\setminus A)-R\).
In a partition with both sides nonempty, each side has rank at least one
and one side has rank \(R\), so \(\lambda_M(A)\ge1\).
If both sides contain at least two elements, simplicity gives rank
at least two on each side, so \(\lambda_M(A)\ge2\).
The size lower bound is at least \(q^{R-2}(q+1)\ge6\).
Thus there are at least four elements and no \(1\)- or \(2\)-separation.
\end{proof}

All deleted points in \eqref{eq:pattern-charts} lie in
\(U=W\oplus X_1\oplus\cdots\oplus X_h\), whose codimension is \(n\).
Lemma~\ref{lem:projective-density-connectivity} applies for \(n\ge2\).
It proves the rank and \(3\)-connectivity assertions for every parameter
used in Theorem~\ref{thm:patterns}.

\subsection{Exact Laurent coefficients and their leading terms}

Fix \(q,S,K\), and let \(n\) vary. Put
\[
 x=q^n,\qquad Y=q^K,\qquad Z=q^R=Yx^D,\qquad
 p_j(x)=\frac{c_j}{u_R}.
\]
For \(n\ge K+2\), these normalized coefficients are finite Laurent
polynomials in \(x\) with rational coefficients. To see this directly,
use \eqref{eq:single-chart} and Theorem~\ref{thm:composition}.
Every variable Gaussian coefficient is a finite product of factors
of the form \(a x^\lambda-1\), divided by fixed nonzero integers, and
\begin{equation}\label{eq:normalized-u}
 \frac{u_{R-d}}{u_R}
   =q^{d(d+1)/2-Kd}x^{-Dd}.
\end{equation}
The relevant individual defect in degree \(j\) has \(d\le j+1\),
by \eqref{eq:single-bose-tau}, independently of \(n\).
All \(m_a=\lambda_an\) are at least \(K+2\), and \(R>2K\), so every
required Gaussian product and kernel is in a fixed valid range.

For \(j\ge0\), put
\[
 v_j=\frac{q^{j(j+1)/2}}{\prod_{a=1}^j(q^a-1)},\qquad v_0=1.
\]
The kernel in \eqref{eq:ambient-kernel} satisfies, for
\(0\le b,j\le K\) and \(R>2K\),
\begin{equation}\label{eq:closed-kernel-coefficients}
 \coeff{j}\mathcal K_{R,b}=
 \begin{cases}
 0,&b<j,\\
 u_{R-2j}\gb{R-j-1}{j},&b=j,\\
 (-1)^{b+j}u_{R-b-j}\gb{R-b}{j},&b>j.
 \end{cases}
\end{equation}
Only the diagonal needs a summation identity. With \(N=R-j\), set
\[
 T_a=(-1)^a\gb{N}{a}u_{N-a},\qquad
 A_a=(-1)^a\gb{N-1}{a}u_{N-a},\qquad A_{-1}=0.
\]
Gaussian Pascal and \(u_{N-a+1}=q^{N-a}u_{N-a}\) give
\(T_a=A_a-A_{a-1}\). Summing through \(a=j\) proves the middle line
of \eqref{eq:closed-kernel-coefficients}; the other lines follow by
extracting the coefficient from the kernel.
Its diagonal consequently has asymptotic form
\begin{equation}\label{eq:kernel-leading}
 \frac{\coeff{j}\mathcal K_{R,j}}{u_R}
      \sim v_j Z^{-j}.
\end{equation}
For \(b>j\) the corresponding normalized coefficient is \(O(Z^{-b})\).
All asymptotics below keep \(q,S,K\) fixed.

Call chart \(a\) active at degree \(j\ge1\) if \(j\le\ell_a+1\).
Its empty-space term with \(b=j\) has positive leading term
\begin{equation}\label{eq:active-leading}
 v_j\gb{\ell_a}{j-1}q^{\ell_a-j+1}
      x^{\lambda_a}Z^{-j}.
\end{equation}
The base Bose term has leading term
\(v_j\gb{K}{j}Z^{-j}\).
The empty-space terms with \(b>j\) are smaller by at least a factor
of order \(Z^{-1}\).

We check that no deficient trace changes the leading terms.
An individual nonzero defect in degree \(j\) has \(d+v=2j\),
\(v\ge j-1\), and \(d\le j+1\). Hence the chart is active.
Its quotient count has degree at most \(j+1\) in \(q^{m_a}=x^{\lambda_a}\),
so its normalized contribution is
\[
 O\bigl(x^{(j+1)\lambda_a}Z^{-2j}\bigr).
\]
If \(\lambda_*\) is the largest weight of an active chart, its exponent
relative to the leading active term is at most
\[
 (j+1)\lambda_a-\lambda_*-Dj
       \le j(\lambda_a-D)<0.
\]
Thus every individual defect is lower order.

For a binary pair term in degree \(j\), the vertex dimension in
\eqref{eq:binary-composition} is \(b=2j-2\).
Its presence requires both vertices to have dimension at least
\(2j-2\ge j-1\), so both charts are active. Its normalized size is
\[
 O\bigl(x^{\lambda_a+\lambda_b}Z^{-2j}\bigr),
\]
and its exponent relative to the leading active term is at most
\[
 \lambda_a+\lambda_b-\lambda_*-Dj
       \le \min(\lambda_a,\lambda_b)-Dj<0.
\]
This includes all overlaps of the nested vertices.
If no chart is active, both individual and mixed deficient corrections
vanish exactly, and only the base Bose leading term remains.
At degree zero there is no deficient contribution; the empty zero-space
contributes one, and all other terms are lower order.

Let \(\lambda(j)\) be the weight of the first active chart at \(j\),
and put \(\lambda(j)=0\) if no chart is active. We have proved
\begin{equation}\label{eq:coefficient-leading}
 p_0(x)\sim1,\qquad
 p_j(x)\sim H_jx^{\lambda(j)-Dj}\quad(1\le j\le K),
\end{equation}
where
\begin{equation}\label{eq:leading-constants}
 H_j=
 \begin{cases}
 v_j\gb{\ell_a}{j-1}q^{\ell_a-j+1-Kj},
     &a\text{ is the first active chart at }j,\\
 v_j\gb{K}{j}q^{-Kj},&\text{no chart is active at }j.
 \end{cases}
\end{equation}
Every \(H_j\) is positive. The weights \(\lambda(j)\) are nonincreasing,
and \(\lambda(1)<D\). Thus the exponents in
\eqref{eq:coefficient-leading}, including the exponent zero at \(j=0\),
strictly decrease. All \(p_j-p_{j+1}\) have positive leading coefficient.
The empty kernels have degree at most \(K\); each individual correction
has degree at most \(\ell_a+1\le K-2\); and each binary pair has degree at
most \(1+\lfloor\min(\ell_a,\ell_b)/2\rfloor\le K-2\).
There is no coefficient above \(K\), and the positive leading coefficient
of \(p_K\) proves eventual exact degree \(K\).

\subsection{Selected and nonselected determinants}

Set \(\lambda_{h+1}=\lambda_{h+2}=0\).
Chart \(a\) stops being active exactly at degree \(s_a\).
At the selected index \(s_a\), the left coefficient has weight
\(\lambda_a\), and the middle coefficient has weight \(\lambda_{a+1}\).
The right coefficient has weight \(\lambda_{a+1}\) if the next selected
index is separated, and weight \(\lambda_{a+2}\) if the next selected
index is consecutive. Consequently the exponent of the Tur\'an ratio is
\[
 2\lambda_{a+1}-\lambda_a-\lambda_{\rm right}.
\]
If the next selected index is separated, this is
\(\lambda_{a+1}-\lambda_a<0\). If it is consecutive, it is \(-1\),
because the triangular weights in \eqref{eq:pattern-parameters}
have second difference one. At the last selected index it is
\(-\lambda_h=-1\). Positive leading constants now show that every
selected ratio tends to zero. Since \(S\) is finite, the conclusion
is simultaneous.

We account separately for every unselected internal index.
At index one the ratio exponent is
\[
 2\lambda_1-\lambda(2)>0.
\]
At an unselected index immediately before a weight drop, the ratio
exponent is \(\lambda_a-\lambda_{a+1}>0\).
At all remaining unselected indices, the three weights are constant.
They come either from one active chart or from the final Bose range.
For \(v\ge2\), set
\[
 f_q(v)=\frac{q^v-1}{q^{v-1}-1}>q.
\]
Adjacent Gaussian ratios and
\[
 \frac{v_j^2}{v_{j-1}v_{j+1}}=\frac{f_q(j+1)}q
\]
give
\begin{equation}\label{eq:constant-weight-ratios}
 \frac{H_j^2}{H_{j-1}H_{j+1}}=
 \begin{cases}
 f_q(\ell_a-j+2)f_q(j)f_q(j+1)/q,
       &\text{inside an active-chart range},\\
 f_q(K-j+1)f_q(j+1)^2/q,
       &\text{inside the final Bose range}.
 \end{cases}
\end{equation}
For the first line all three degrees are active for the same chart,
and \(j\ge2\), so every argument of \(f_q\) is at least two.
For the second line \(j<K\), with the same property.
Both ratios are strictly greater than \(q^2\).
Thus all unselected internal determinants have positive leading
coefficients, while the negatives of all selected determinants have
positive leading coefficients. This classification also covers the
last internal index and consecutive selected indices.

\subsection{Effective thresholds and a uniform bound}

To obtain one bound valid for every stated conclusion, consider the
following Laurent polynomials
\begin{equation}\label{eq:sign-certificates}
 p_K,\qquad p_j-p_{j+1}\ (0\le j<K),\qquad
 \begin{cases}
 p_{j-1}p_{j+1}-p_j^2,&j\in S,\\
 p_j^2-p_{j-1}p_{j+1},&j\notin S,
 \end{cases}
 \quad(1\le j<K).
\end{equation}
Each has a positive leading coefficient. For any one of them, write
\(p(x)=\sum_{e\le d}a_ex^e\), \(a_d>0\), and define
\[
 B_p=1+\frac{\sum_{e<d}|a_e|}{a_d}.
\]
For \(x\ge B_p\), integer exponents and \(x\ge1\) give
\[
 p(x)\ge x^{d-1}
          \left(a_dx-\sum_{e<d}|a_e|\right)>0.
\]
Thus the exact coefficients in the preceding section give an effective
rational threshold for all conclusions.

To bound that threshold uniformly, let \(\|p\|_1\) be the sum of the
absolute values of the Laurent coefficients of \(p\).
This norm is subadditive and submultiplicative, independently of the
range of the exponents. Throughout,
\[
 K\ge3,\qquad h\le K-2,\qquad
 \ell_a\le K-3,\qquad q\ge2.
\]
We use the elementary bounds
\(\gb{a}{b}\le q^{ab}\), \(K+1\le q^K\), and \(2\le q\).

First, a normalized kernel coefficient has norm at most
\begin{equation}\label{eq:kernel-height}
 \left\|\frac{\coeff{j}\mathcal K_{R,b}}{u_R}\right\|_1
       \le q^{K^2+3K},\qquad 0\le b\le K.
\end{equation}
Indeed, the variable Gaussian products in the ambient kernel have at
most \(K\) factors, each numerator of norm at most
\(1+q^K\le q^{K+1}\), and denominators at least one.
For \(0\le d\le2K\), the scalar exponent in
\eqref{eq:normalized-u} satisfies
\[
 d(d+1)/2-Kd\le K;
\]
the convex quadratic has maximum \(K\) at an endpoint of this interval.
There are at most \(K+1\) summands. Multiplication of these bounds proves
\eqref{eq:kernel-height}.

For the base and the additional empty terms, the fixed Gaussian
multiplier is at most \(q^{K^2}\), and the chart multiplier is at most
\(q^{K^2+K}\). Also \(\|x^{\lambda_a}-1\|_1=2\le q\).
Summing at most \(K+1\) kernels for each of \(h\) charts and the base
gives the bound
\begin{equation}\label{eq:empty-height}
 (h+1)q^{2K^2+5K+1}.
\end{equation}

Next consider the individual deficient terms in \eqref{eq:single-chart}.
In the Gaussian product for \eqref{eq:single-deficiency-count}, every
variable numerator has norm \(1+q^{-a}\le2\le q\), and the additional
scalar \(q^{1-d}\) is at most one. Its norm is at most \(q^K\).
The extra binary rank-two term has norm at most \(2q^K\); the combined
quotient count therefore has the safe bound \(q^{K+2}\).
The lift Gaussian, lift scalar and tau in
\eqref{eq:single-bose-tau} are each at most \(q^{K^2}\).
For tau this uses its nonnegativity and the upper bound
\(\tau\le\gb{v}{s-1}\). There are at most \(K^2\le q^{2K}\)
parameter pairs \((d,v)\), and the normalized \(u\)-factor contributes
at most \(q^K\). All individual defects together consequently have
norm at most
\begin{equation}\label{eq:individual-height}
 hq^{3K^2+4K+2}.
\end{equation}

Finally, in the binary pair sum, bound the Gaussian by \(q^{K^2}\),
the vertex scalar by \(q^{2K}\), and the affine tau from
\eqref{eq:affine-tau} by \(q^{K^2+K}\).
The two binomials in \(x\) have product norm \(4\le q^2\).
The normalized \(u\)-factor is at most \(q^K\); summing vertex dimensions
costs at most \(K\le q^K\), and there are fewer than \(h^2\) pairs.
The pair terms thus have norm at most
\begin{equation}\label{eq:pair-height}
 h^2q^{2K^2+5K+2}.
\end{equation}
For \(q>2\) these terms are absent.

Each exponent in \eqref{eq:empty-height}--\eqref{eq:pair-height}
is at most \(3K^2+6K+3\), and
\((h+1)+h+h^2=(h+1)^2\).
Using \(h+1\le q^K\) and
\(3K^2+8K+3\le6K^2\), valid for all \(K\ge3\), we obtain
\begin{equation}\label{eq:coefficient-height}
 \|p_j\|_1\le q^{6K^2}\qquad(0\le j\le K).
\end{equation}

Every leading constant in \eqref{eq:leading-constants} satisfies
\(H_j\ge q^{-K^2}\): \(v_j>1\) for \(j\ge1\), the relevant Gaussian
is at least one, the extra chart scalar is at least one, and
\(q^{-Kj}\ge q^{-K^2}\). Set \(H_0=1\).
For a determinant with unequal leading exponents, the positive leading
coefficient of its signed version in \eqref{eq:sign-certificates} is
a square or a product of two such constants, at least \(q^{-2K^2}\).
When the exponents agree, \eqref{eq:constant-weight-ratios} gives the bound
\[
 H_j^2-H_{j-1}H_{j+1}
   \ge(1-q^{-2})H_j^2\ge\frac34q^{-2K^2}.
\]
The terminal and adjacent-difference polynomials have the stronger
leading-coefficient bound \(q^{-K^2}\).
Every polynomial in \eqref{eq:sign-certificates} therefore has leading
coefficient at least \(\frac34q^{-2K^2}\), and, by
\eqref{eq:coefficient-height}, norm at most \(2q^{12K^2}\).
It follows that
\begin{equation}\label{eq:uniform-cauchy-bound}
 B_p\le1+\frac83q^{14K^2}\le q^{14K^2+2}.
\end{equation}
Since \(20K^2>14K^2+2\) and \(20K^2\ge K+2\), every \(n\ge20K^2\)
lies in the stable range and satisfies \(q^n>B_p\) for every required
polynomial.

\begin{proof}[Proof of Theorem~\ref{thm:patterns}]
The construction and Lemma~\ref{lem:projective-density-connectivity}
give simplicity, rank \(R\), representability and \(3\)-connectivity.
The exact single-chart and pair formulas show that no coefficient
above \(K\) occurs. The bound \eqref{eq:uniform-cauchy-bound} makes
every polynomial in \eqref{eq:sign-certificates} positive for the stated
range of \(n\). Positivity of \(p_K\), followed by the adjacent
differences, gives the actual degree and all strict coefficient
inequalities. The signed determinants give exactly the prescribed
negative set and strictly positive determinants elsewhere.
The selected leading exponents already prove
\eqref{eq:simultaneous-ratio-limit}. Finally, set \(n=20K^2\) in
\(R=K+Dn\) to obtain \eqref{eq:pattern-rank-bound}.
\end{proof}

\subsection{Severity and prescribed depth}

\begin{corollary}\label{cor:simultaneous-severity}
Fix \(q,S,K\) as in Theorem~\ref{thm:patterns}.
For every \(\varepsilon>0\), the matroids in that theorem satisfy
\[
 0<\frac{c_j^2}{c_{j-1}c_{j+1}}<\varepsilon
       \qquad(j\in S)
\]
for all sufficiently large \(n\), simultaneously, while preserving
the exact degree, strict coefficient decrease and all prescribed signs.
\end{corollary}

\begin{proof}
Every denominator is positive in the theorem's range.
Each selected ratio tends to zero, and there are finitely many of them.
Take the maximum of their thresholds and \(20K^2\).
\end{proof}

\begin{corollary}\label{cor:prescribed-depth}
For every finite field \(\F_q\), every \(i\ge2\) and every \(d\ge1\),
there is a simple \(3\)-connected \(\F_q\)-representable matroid with
actual inverse-Q degree \(i+d\), positive strictly decreasing
coefficients, a negative determinant exactly at index \(i\), and
positive determinants at every other internal index.
Such a matroid can be chosen in rank
\[
 (i+d)+40(i+d)^2.
\]
The failure ratio at the fixed index can also be made arbitrarily small
without changing the polynomial degree.
\end{corollary}

\begin{proof}
Take \(S=\{i\}\), \(K=i+d\), so \(D=2\), and first choose
\(n=20K^2\). For the last assertion let \(n\) grow and apply
Corollary~\ref{cor:simultaneous-severity}.
\end{proof}

For a rational \(\alpha=a/b\in(0,1)\), take \(i=am\) and
\(d=(b-a)m\), with \(m\) large enough that \(i\ge2\).
Corollary~\ref{cor:prescribed-depth} then gives
\(i/\deg Q=\alpha\), while both \(i\) and \(\deg Q-i\) tend to infinity.
For an arbitrary real \(\alpha\in(0,1)\), integer choices with
\(i/(i+d)\to\alpha\) give the corresponding limiting statement.
The severity parameter may be chosen separately for each member, so the
ratios can tend to zero along these sequences as well.
Thus no nonempty open interval of normalized internal positions is
universally protected from ordinary-Q log-concavity failure.
The exact-pattern theorem also realizes any finite consecutive block
of indices at least two. All these conclusions retain positive
determinants at index one.

\section{The first-index boundary}\label{sec:first-index-boundary}
Theorem~\ref{thm:patterns} prescribes negative determinants only at
indices at least two. Its examples satisfy a stronger inequality at
index one.

\begin{corollary}\label{cor:first-index-factor-two}
For every matroid in Theorem~\ref{thm:patterns},
\[
 c_1^2>2c_0c_2.
\]
In particular, the first index is internal and $\Delta_1>0$.
\end{corollary}
\begin{proof}
Use the Laurent polynomials $p_j=c_j/u_R$ and weights from
Section~\ref{sec:patterns}. All charts are active at degree one, so
$\lambda(1)=\lambda_1\ge1$ and $0\le\lambda(2)\le\lambda_1$.
The leading exponent of $p_1^2$ therefore exceeds that of $p_0p_2$
by $2\lambda_1-\lambda(2)>0$.
Thus $f=p_1^2-2p_0p_2$ has positive leading coefficient
$H_1^2\ge q^{-2K^2}$. By \eqref{eq:coefficient-height},
$\|f\|_1\le3q^{12K^2}$. The same Cauchy bound used above gives
\[
 B_f\le1+3q^{14K^2}\le q^{14K^2+2}<q^{20K^2}.
\]
Consequently $f(q^n)>0$ whenever $n\ge20K^2$.
Multiplication by $u_R^2>0$ gives the result. The theorem has
$K\ge3$ and $c_0,c_1,c_2>0$, so the inequality is internal.
\end{proof}

This argument applies to the construction over every finite field,
including $\F_2$. It does not establish $\Delta_1(M)\ge0$ for an arbitrary
matroid whose inverse polynomial has degree at least two. That universal
question remains open. Positive-class results preserved in the
mathematical supplement are not used here. The examples in this paper
are strictly decreasing and hence unimodal; they do not settle
unimodality of ordinary $Q$ in general.

\section{Verification and scope}\label{sec:methods}
The proofs are symbolic. The verification materials provide
error-detection checks and alternative calculations, using integer or
rational arithmetic for every decisive computation.

Two coefficient implementations check the all-field chart and pattern
formulas, actual support, determinant signs and threshold estimates.
They share the ambient reconstruction identity and trace classification.
Literal flat-lattice calculations on smaller represented matroids use
characteristic recursion for $P$ and signed incidence inversion for $Q$.
They include binary and ternary cases, arithmetic in $\F_4$, affine
top-coefficient controls, and a case where dependent external directions
invalidate the composition formula. Arithmetic modulo four is not used
as a substitute for $\F_4$. The large matroids in
Theorem~\ref{thm:patterns} are covered by its proof, not enumeration of
their flat lattices.

Whole-target characteristic-recursion calculations also recover the
$952$- and $889$-element binary examples and the $3025$-element ternary
example. These calculations use neither the Bose coefficient sum nor
sequential deletion. They share finite-field geometry, proved subspace
multiplicities, simplification invariance and the ordinary-$Q$
conventions. Uniform, fan, thagomizer and paving formulas give further
small-case controls \cite{GX,GLiLi,GLX,XZ}. Agreement between programs
does not replace a proof of their shared geometric assumptions.

The Lean 4.12.0 files check specified algebraic identities, conditional
determinant inequalities and binary coordinate geometry, including
column distinctness, closure exchange and quotient survival. Their scope
and axiom reports are included. They do not identify the coefficient
formulas with $Q_M$ in Lean or formalize the locality, composition and
pattern theorems. Those steps are supplied by the mathematical proofs.

The inverse-$Z$ polynomial $Y$ is a different invariant. For example,
\cite[Theorem 2.2]{BFMN} gives
\[
 Y_M(t)=\sum_{F\in L(M)}(-1)^{\rk M-\rk F}\mu(M/F)
          t^{\rk M-\rk F}Q_{M|F}(t).
\]
A failure for $Q_M$ alone does not imply a failure for $Y_M$.
No global minimum size, optimal rank, regular or characteristic-zero
representability, or universal first-index result is claimed here.

\subsection{AI assistance}
OpenAI Codex, using GPT-6 Astra, assisted with mathematical exploration,
proof development and checking, exact computation and research code,
literature analysis, and manuscript drafting. The author reviewed and
edited the mathematical arguments and manuscript and takes responsibility
for its contents.

\subsection{Availability of verification materials}\label{sec:availability}
The reproducibility materials are publicly available on Figshare at DOI
\href{https://doi.org/10.6084/m9.figshare.33472711}{\nolinkurl{10.6084/m9.figshare.33472711}}.
Archive version 2.0.0 in the second public version of the Figshare item
contains the mathematical supplement, verification programs, data,
scoped Lean developments, expected outputs and reproduction instructions.
It also contains a first-index subpackage with positive-class proofs
and catalogue methods. Manifests and an external checksum record file
integrity. The claim-to-file map distinguishes formula checks,
alternative reconstructions and literal small-case controls.

Statement numbers in that archive refer to the preceding manuscript
layout: its first-index positive-class statements remain in the
mathematical supplement. Corollary~\ref{cor:first-index-factor-two}
is proved here from the coefficient and leading-term bounds of
Section~\ref{sec:patterns}; it requires no additional verification
program in the public archive.

\end{document}